\documentclass{amsart}
\usepackage{amsmath,amsthm,amsfonts}
\usepackage[colorlinks=false,linktocpage=true]{hyperref}
\numberwithin{equation}{section}
\font\tenscrpt=eusm10
\font\sevenscrpt=eusm10 scaled 700
\font\fivescrpt=eusm10 scaled 500
\newfam\eusmfam
\textfont\eusmfam=\tenscrpt
\scriptfont\eusmfam=\sevenscrpt
\scriptscriptfont\eusmfam=\fivescrpt
\def\scrpt#1{{\fam\eusmfam\relax#1}}
\newtheorem{thm}{Theorem}[section]

\newtheorem{lem}{Lemma}[section]

\newtheorem{rem}{Remark}[section]

\newcommand{\thmref}[1]{Theorem~{\rm $\ref{#1}$}}

\newcommand{\lemref}[1]{Lemma~{\rm $\ref{#1}$}}

\newcommand{\remref}[1]{Remark~{\rm $\ref{#1}$}}
\newcommand{\eqnref}[1]{{\rm (\ref{#1})}}
\def\eqdef{\overset{\triangle}{=}}

\def\dstyle{\displaystyle}

\def\dstyle{\displaystyle}

\def\eqdef{\overset{\triangle}{=}}
\def\utx{u(t,x)}
\def\vsx{v(s,x)}
\def\vsxe{v_\epsilon(s,x)}

\def\utxe{u_\epsilon(t,x)}

\def\Xx{X^x}

\def\RBM{|B(t)|}

\def\BTP{{\mathbb X}^x_{B}(t)}
\def\BTPr{{\mathbb X}^x_{B}(r)}

\def\eBTP{{\mathbb X}^x_{\epsilon B}(t)}
\def\BTPz{{\mathbb X}^0_{B}(t)}
\def\BTPez{{\mathbb X}^0_{\epsilon B}(t)}

\def\EBTP{{\mathbb X}^x_{B,e}(t)}
\def\EeBTP{{\mathbb X}^x_{\epsilon B,e}(t)}

\def\kEBTP{{\mathbb X}^{x,k}_{B,e}(t)}
\def\kEeBTP{{\mathbb X}^{x,k}_{\epsilon B,e}(t)}
\def\kEBTPr{{\mathbb X}^{x,k}_{B,e}(r)}
\def\oEBTP{{\mathbb X}^{x,1}_{B,e}(t)}

\def\P{{\mathbb P}}

\def\EP{{\mathbb E}_{\P}}

\def\N{{\mathbb N}}
\def\R{{\mathbb R}}
\def\Rd{{\mathbb R}^d}
\def\Rp{{\mathbb R}_+}
\def\OFP{(\Omega,{\scrpt F},\{\Filt\},\P)}
\def\Filt{{\scrpt F}_t}
\def\SG{\scrpt T}
\def\GN{\scrpt A}

\def\eil1{e_i^\clubsuit(1)}

\def\el1{e^\clubsuit_1}
 \newcounter{crcounter}

\begin{document}
\title[BTP and fourth order PDEs II]{Brownian-Time Processes: The PDE Connection II and the Corresponding Feynman-Kac Formula}
\author{Hassan Allouba}
\thanks{Supported in part by NSA grant MDA904-02-1-0083}
\address{Department of Mathematics,
Indiana University, Bloomington, IN 47405-7106}
\urladdr{http://php.indiana.edu/\~{ }allouba}
\email{allouba@indiana.edu}
\date{November 15, 2001.}
\begin{abstract}
We delve deeper into our study of the connection of Brownian-time processes (BTPs) to fourth order parabolic PDEs, which we introduced in a recent joint
article with W.~Zheng.   Probabilistically,
BTPs and their cousins BTPs with excursions form a unifying class of interesting stochastic processes that includes the celebrated IBM of Burdzy and
other new intriguing processes, and is also connected to the Markov snake of Le Gall.   BTPs also offer a new connection of probability to PDEs that is
fundamentally different from the  Markovian one.    They solve fourth order PDEs in which the initial function plays an important role in the PDE itself,
not only as initial data.    We connect  two such types of interesting  and new PDEs to BTPs.  The first is obtained by running the BTP and then integrating along its path, and the
second type of PDEs is related to what we call the Feynman-Kac formula for BTPs.  A special case of the second type is a step towards a probabilistic solution
to linearized Cahn-Hilliard and Kuramoto-Sivashinsky type PDEs, which we  tackle in an upcoming paper.
\end{abstract}
\maketitle
 \section{Introduction and statements of results}
 Let $B(t)$ be a one-dimensional Brownian motion starting at $0$ and $X^x(t)$
be an independent $\Rd$-valued continuous Markov process started at $x$, both defined on
a probability space $\OFP$.  We call the process
$\BTP\eqdef X^x(|B(t)|)$
a Brownian-time process (BTP): the regular clock $t$ is replaced with the Brownian clock $|B(t)|$.  In the special case where $X^x$ is a Brownian
motion starting at $x$ we call the process $\BTP$  a Brownian-time Brownian
motion (BTBM).  Excursions-based Brownian-time processes (EBTPs) are obtained
from BTPs by breaking up the path of $\RBM$ into excursion
intervals---maximal intervals $(r,s)$ of time on which $\RBM>0$---and, on each
such interval, we pick
an independent copy of the Markov process $\Xx$ from a finite or an infinite
collection.   Frequently  in applied PDEs (like the Allen-Cahn and Cahn-Hilliard and others), an order parameter $\epsilon$ with some physical significance
is an important part of the PDE; and to accomodate such a parameter, we introduce the $\epsilon$-scaled BTPs $\eBTP\eqdef\Xx(\epsilon|B(t)|)$ and their
excursion cousins (see \thmref{PDE2th} below for a PDE connection).

BTPs and EBTPs may be regarded as canonical constructions for several famous as well as interesting new processes.  To see this,
observe that the following processes have the one dimensional distribution $\P(\BTP\in dy)$:
\begin{enumerate}
\renewcommand{\labelenumi}{(\alph{enumi})}
\item Markov snake---when $|B(t)|$ increases we generate a new
independent path.  See Le Gall (\cite{LeGall1}, \cite{LeGall2}, and \cite{LeGall3})
for applications to the nonlinear PDE $\Delta u=u^2$.
\item Let $X^{x,1}(t), \ldots,X^{x,k}(t)$ be independent copies of
$\Xx(t)$ starting from point $x$.  On each excursion interval of $\RBM$ use one of the
$k$ copies chosen at random.  When $x=0$, $X^x$ is a Brownian motion starting
at $0$, and $k=2$ this reduces to the iterated Brownian motion or IBM (see Burdzy \cite{Burdzy1, Burdzy2} Burdzy et al.~\cite{BD1,BD2} and
Khoshnevisan et al.~\cite{DL}).  We identify such a process by
the abbreviation $k$EBTP and we denote it by $\kEBTP$.  Of course, when $k=1$ we obtain a BTP.
\item Use an independent copy of $X^x$ on each excursion interval of $\RBM$.
This is the $k\to\infty$ weak limit of (b) (for a rigorous statement and proof, see the
Appendix of \cite{AZ}).  It is
intermediate between IBM and the Markov snake.
Here, we go forward on a new independent path only after $|B(t)|$
reaches $0$.  This process is abbreviated as EBTP and is denoted by $\EBTP$.
\end{enumerate}
As in the case of standard Brownian motion (and more generally diffusions), there is a host of interesting connections of BTPs to PDEs.  However, unlike
the Brownian motion's link to PDEs (see e.g., \cite{Ba1,Ba2,D,KS}), the PDEs here are fourth order and they are distinguished by the feature that the initial
function is a fundamental part of the PDE itself, not only as initial data.   We call such PDEs initially perturbed.

In this paper we always assume that the generator $\GN$ of the outer Markov process $\Xx$, and its associated semigroup $\SG_s$, satisfy the property
\begin{equation}
\begin{split}
&f :\Rd\to\R\mbox{ bounded and } D_{ij} f \mbox{ is H\"older continuous }\forall\ 1\le i,j\le d\\
&\implies \frac{\partial^2\SG_s f(x)}{\partial s^2}=\GN^2\SG_s f \mbox{ is continuous on }(0,\infty)\times\Rd,  \mbox{ and }\\
&\begin{cases}
(a)\ \dstyle\GN^2\int_0^\infty\SG_sf(x) p_t(0,s) ds=\int_0^\infty \GN^2\SG_sf(x)p_t(0,s) ds; &\\
(b)\ \dstyle\GN^2\int_0^t\int_0^\infty \SG_s f(x)p_r(0,s) dsdr=\int_0^t\int_0^\infty \GN^2\SG_s f(x)p_r(0,s) dsdr; &
\end{cases}
\end{split}
\tag{P}
\label{P}
\end{equation}
where $p_t(0,s)$ is the transition density of the Brownian motion $B(t)$ and $D_{ij}f$ is $\partial^2/\partial x_i \partial x_j$.  Property \eqref{P} is satisfied
when $\Xx$ is a Brownian motion  (see  \lemref{diffunderint} below).

The first theorem gives us the fourth order PDE solved by running a Brownian-time process and then averaging the sum of $f(\BTP)$ and the integral of a
function $g$ along the path of $\BTP$.
 \begin{thm}\label{BTPPDES1}
Let $\SG_sf(x)=\EP f(\Xx(s))$ be the semigroup of the continuous Markov process $\Xx(t)$ and
$\GN$ its generator.  Let $f$ and $g$ be bounded continuous functions in ${\mathbb D}(\GN)$, the domain of $\GN$, such that $D_{ij} f$ and $D_{ij} g$ are
bounded and H\"older continuous with exponent $0<\alpha\le1$, for all $1\le i,j\le d$.
If
\begin{equation}
u(t,x)=\EP\left[f(\kEBTP)+\int_0^tg(\kEBTPr)dr\right]
\label{sol1}
\end{equation}
for any $k\in\N$ $($as stated before $\oEBTP=\BTP$\/$)$, or if we replace $\kEBTP$ with $\EBTP$ in \eqref{sol1}, then $u$ solves the PDE
\begin{equation}
\begin{cases}
 \dfrac{\partial}{\partial t} u(t,x)=\dfrac{\GN f(x)}{\sqrt{2\pi t}}+\dfrac{\sqrt{2 t}}{\pi}\GN g(x)+\dfrac12\GN^2\utx;
 &t>0,\,x\in\Rd,
\cr u(0,x)=f(x)={\dstyle\lim_{\substack{t\downarrow0\\y\to x}}} u(t,y); &x\in\Rd,
\end{cases}
\label{BTPPDE}
\end{equation}
where the operator $\GN$ acts on $u(t,x)$ as a function of $x$ with $t$ fixed.
In particular, if $\BTP$ is a BTBM and
$\Delta$ is the standard Laplacian, then $u$ solves
\begin{equation}
\begin{cases}
  \dfrac{\partial}{\partial t}u(t,x)=\dfrac{\Delta f(x)}{\sqrt{8\pi t}}+\dfrac{\sqrt{2 t}}{2\pi}\Delta g(x)+\dfrac18\Delta^2u(t,x);& t>0,\,x\in\Rd,
\cr u(0,x)=f(x)={\dstyle\lim_{\substack{t\downarrow0\\y\to x}}} u(t,y); & x\in\Rd.
\end{cases}
\label{BTBMPDE}
\end{equation}
\end{thm}
\begin{rem}
The inclusion of the initial function $f(x)$ in the PDEs  \eqref{BTPPDE}  and \eqref{BTBMPDE} is a reflection of the non-Markovian property of our BTP.
Thus, as mentioned above, the role of $f$ here is fundamentally different from its role in the standard Markov-PDE connection.  Moreover, as $t$ gets large, we see that the effect
of the initial function $f$, through $\GN f$, fades away at the rate $1/\sqrt{2\pi t}$; while the effect of $g$, through $\GN g$, becomes more dominant at a
rate $\sqrt{2 t}/ \pi$.    We also remark that property \eqnref{P} $($excluding part $(b)\/)$ should have been explicitly assumed for the case of general outer Markov process $\Xx$
in Theorem 0.1 in \cite{AZ}.
\end{rem}
Next, we solve the PDE obtained by running an $\epsilon$-scaled BTP and averaging the product of $f(\eBTP)$ with the negative exponential of
$|B(t)|/\epsilon$ (the Brownian clock speeded up by $1/\epsilon$).    When $\epsilon=1$ this is a special case of the Feynman-Kac formula for BTP
given by \eqref{udefn1} in \thmref{FK}.   However, it deserves to be singled out, for it is a first step towards the probabilistic study of linearized
Cahn-Hilliard and Kuramoto-Sivashinsky type PDEs, which we undertake in an upcoming paper \cite{AKSCH} (see also \remref{CHrem} below).
\begin{thm}\label{PDE2th}
Under the same conditions on $f$ as in \thmref{BTPPDES1}, and for $\epsilon> 0$, if
\begin{equation}
\utxe\eqdef \EP\left[f(\kEeBTP)\exp\left(-{\frac{|B(t)|}{\epsilon}}\right)\right]
\label{sol2}
\end{equation}
for any $k\in\N$, or if we replace $\kEeBTP$ with $\EeBTP$ in \eqref{sol2}, then $u$ solves
\begin{equation}
\begin{cases}
 \dfrac{\partial}{\partial t} \utxe=\dfrac{1}{\sqrt{2\pi t}}\left[\epsilon\,\GN f(x)-\dfrac{1}{\epsilon}f(x)\right]\\
 \hspace{1.7cm}+ \dfrac{1}{2\epsilon^2}\utxe-\GN \utxe + \dfrac{\epsilon^2}{2}\GN^2\utxe;&{ t>0,\,x\in\Rd,}
\cr u_\epsilon(0,x)=f(x)={\dstyle\lim_{\substack{t\downarrow0\\y\to x}}} u_\epsilon(t,y); &x\in\Rd.
\end{cases}
\label{PDE2}
\end{equation}
In particular, if the outer Markov process $\Xx$ in \eqref{sol2} is a Brownian motion then $\utxe$ solves
\begin{equation}
\begin{cases}
 \dfrac{\partial}{\partial t} \utxe=\dfrac{1}{\sqrt{2\pi t}}\left[\dfrac{\epsilon}{2}\,\Delta f(x)-\dfrac{1}{\epsilon}f(x)\right]
 \\ \hspace{1.7cm} + \dfrac{1}{2\epsilon^2}\utxe -\dfrac12\Delta \utxe+ \dfrac{\epsilon^2}{8}\Delta^2\utxe;
 &t>0,\,x\in\Rd,
\cr u_\epsilon(0,x)=f(x)={\dstyle\lim_{\substack{t\downarrow0\\y\to x}}} u_\epsilon(t,y); &x\in\Rd.
\end{cases}
\label{BMPDE2}
\end{equation}
\end{thm}
\begin{rem}
This time the initial function $f$ affects our PDE through both $f$ and $\GN f;$ and, as before, these effects diminish as $t$ grows larger at the rate $1/\sqrt{2\pi t}$.
Also, for small $\epsilon$,
we see that the effects $f$ and $u_\epsilon$ are larger and eventually, as $\epsilon\searrow0$, $u_\epsilon$ dominates all other terms in the PDE.  We also
comment briefly that, although a certainly different PDE, the last two terms in \eqref{BMPDE2}
$($the bi-Laplacian and the Laplacian of the solution $u_\epsilon$\/$)$ look like those in a linearized Cahn-Hilliard equation with the
correct $\epsilon$-scaling, albeit with the opposite sign for $\Delta^2$.
\label{CHrem}
\end{rem}
The next result gives a Feynman-Kac type formula for BTP's and connect it to fourth order PDEs:
\begin{thm}\label{FK}
Assume that $f,c:\Rd\to\R$ are bounded, $c\le0$, and $D_{ij}f$ and $D_{ij}c$ are bounded and H\"older continuous with exponent $0<\alpha\le1$, for all
$1\le i,j\le d$.  If the $|D_{ij} v(s,x)|\le K_T$ $\forall (s,x)\in[0,T]\times\Rd$, for any time $T>0$, for all $i,j$, where $K_T>0$ is a constant depending only
on $T$ and
\begin{equation}
\vsx\eqdef\EP\left[f(\Xx(s))\exp\left(\int_0^{s}c(\Xx(r))dr\right)\right],
\label{vdefn1}
\end{equation}
and where $\Xx$ is a $d$-dimensional Brownian motion starting at $x$ under $\P$.  Then,
\begin{equation}
\utx\eqdef\EP\left[f(\BTP)\exp\left(\int_0^{|B(t)|}c(\Xx(r))dr\right)\right]
\label{udefn1}
\end{equation}
solves
\begin{equation}
\begin{cases}
 \dfrac{\partial}{\partial t} \utx=\dfrac{1}{\sqrt{2\pi t}}\left[\dfrac{1}{2}\Delta f(x)+c(x)f(x)\right]\vspace{1mm}
 \\ \hspace{1.7cm} + \left[\dfrac14\Delta c(x) +\dfrac12c^2(x)\right]\utx +\dfrac12\nabla c(x)\cdot\nabla\utx\vspace{1mm}
 \\ \hspace{1.7cm} +\dfrac12c(x)\Delta\utx+ \dfrac{1}{8}\Delta^2\utx;
 &t>0,\,x\in\Rd,\vspace{1mm}
\cr u_\epsilon(0,x)=f(x)={\dstyle\lim_{\substack{t\downarrow0\\y\to x}}} u_\epsilon(t,y); &x\in\Rd.
\end{cases}
\label{BTPFKPDE}
\end{equation}
\end{thm}
\begin{rem}
As with previous PDEs, the effect of the initial function $f$---this time through $\frac{1}{2}\Delta f(x)+c(x)f(x)$---fades away as $t$ grows larger at the
BTP rate of $1/\sqrt{2\pi t}$.   Another
feature of the BTP Feynman-Kac PDE that is quite different from the standard Feynman-Kac PDE for Brownian motion is the existence
of the interaction term $c(x)f(x)$ between the initial function $f$ and the function $c$.  Also,
We suspect that the conditions on $f$ and $c$ in \thmref{FK} above are sufficient to imply the condition
$|D_{ij} v(s,x)|\le K_T$ $\forall (s,x)\in[0,T]\times\Rd$, for any time $T>0$, for all $i,j$, but we do not have a proof of this yet.
Finally, it is worth emphasizing that the BTP solutions to the PDEs presented in this article are all bounded.
\label{FKrem}
\end{rem}
 \section{Proofs of results}
\subsection{A Technical Lemma}
We start with a differentiating-under-the-integral type lemma
\begin{lem}
Let $\Xx$ be a $d$-dimensional Brownian motion starting at $x$ under $\P;$ and let  $f,g:\Rd\rightarrow\R$ be bounded and measurable such that $D_{ij} f$ and $D_{ij}g$ are H\"older continuous, with exponent $0<\alpha\le1$, for
$1\le i,j\le d$.  Let
\begin{equation}
\begin{split}
u_1(t,x)&\eqdef\int_0^\infty \EP f(\Xx(s)) p_t(0,s) ds\\
u_2(t,x)&\eqdef\int_0^t\int_0^\infty\EP g(\Xx(s))p_r(0,s) ds dr,
\end{split}
\label{diffunderint}
\end{equation}
then  $\Delta^2u_1(t,x)$ and $\Delta^2u_2(t,x)$ are finite and
\begin{equation}
\begin{split}
\Delta^2u_1(t,x)&=\int_0^\infty\Delta^2 \EP f(\Xx(s)) p_t(0,s) ds\\
\Delta^2u_2(t,x)&=\int_0^t\int_0^\infty\Delta^2\EP g(\Xx(s))p_r(0,s) ds dr.
\end{split}
\label{diffunderint1}
\end{equation}
If we additionally assume the $D_{ij}f$ and $D_{ij}g$ are bounded $($for all $i,j$\/$)$, then $\Delta^2u_1(t,x)$ and $\Delta^2u_2(t,x)$ are continuous
on $(0,\infty)\times\Rd$.
\end{lem}
\begin{proof}
For notational simplicity we show that 
\begin{equation}
\begin{cases}
(a)\ \dstyle{\frac{\partial^4u_1}{\partial x_i^4}=\int_0^\infty \frac{\partial^4}{\partial x_i^4}\EP f(\Xx(s)) p_t(0,s) ds}; &i=1,\ldots,d,\\
(b)\ \dstyle{\frac{\partial^4u_2}{\partial x_i^4}=\int_0^t\int_0^\infty \frac{\partial^4}{\partial x_i^4}\EP g(\Xx(s)) p_r(0,s) dsdr}; &i=1,\ldots,d,
\end{cases}
\label{toshow}
\end{equation}
the mixed derivatives cases follow the same steps.
In the remainder of the proof, fix an arbitrary $i\in\{1,\ldots d\}$.  We start with assertion (a) in \eqref{toshow}.  Using the boundedness on $f$ and
Problem 3.1 p.~254 in \cite{KS} (the case $\Rd$ with $d>1$ is a simple extension when $f$ is bounded), the symmetry of  $p^{(d)}_s(x,y)$ (the density of $\Xx$) in $x$ and $y$, and the facts that
$$\lim_{y_i\to\pm\infty}f(y)\frac{\partial^3}{\partial y_i^3}p^{(d)}_s(x,y)=\lim_{y_i\to\pm\infty}\frac{\partial}{\partial y_i}f(y)\frac{\partial^2}{\partial y_i^2}p^{(d)}_s(x,y)=0,$$
(since $f$ is bounded and $\frac{\partial}{\partial y_i}f(y)$ is Lipschitz in $y_i$), we get
\begin{equation}
\begin{split}   \label{fstep}
&\frac{\partial^4}{\partial x_i^4}\EP f(\Xx(s)) p_t(0,s)=
\left(\int_{\Rd}f(y)\frac{\partial^4}{\partial x_i^4}p^{(d)}_s(x,y)dy\right)p_t(0,s)\\
&= \left(\int_{\Rd}f(y)\frac{\partial^4}{\partial y_i^4}p^{(d)}_s(x,y)dy\right)p_t(0,s)\\&=
\left(\int_{\Rd}\frac{\partial^2}{\partial y_i^2}f(y)\frac{\partial^2}{\partial y_i^2}p^{(d)}_s(x,y)dy\right)p_t(0,s).
\end{split}
\end{equation}
Rewriting the last term in \eqref{fstep}, and letting $h_i(y)\eqdef\partial^2 f(y)/\partial y_i^2$, we have
\begin{equation}
\begin{split}
&\frac{1}{\sqrt{2\pi t}}\left(\int_{\Rd}(2\pi s)^{-{d/2}}\left(\frac{(x_i-y_i)^2-s}{s^2}\right)e^{-|x-y|^2/2s}h_i(y)dy\right)e^{-s^2/2t}\\
&= \frac{1}{\sqrt{2\pi t}}e^{-s^2/2t}\EP\left[\left(\frac{(x_i-X_i^x(s))^2-s}{s^2}\right)h_i(\Xx(s))\right]\\
&= \frac{1}{\sqrt{2\pi t}} e^{-s^2/2t}\EP\left[\left(\frac{(x_i-X_i^x(s))^2-s}{s^2}\right)\left(h_i(\Xx(s))-h_i(x)\right)\right],
\end{split}
\label{sstep}
\end{equation}
where we used the fact that $\EP\left({(x_i-X_i^x(s))^2-s}\right)=0$ to obtain the last equality.   Now, using the Brownian motion scaling,
we have
\begin{equation}
\EP\left|(x_i-X_i^x(s))^2-s\right|^2=s^2\EP\left|\left(\frac{X_i^0(s)}{\sqrt{s}}\right)^2-1\right|^2=s^2\EP\left|\left({X_i^0(1)}\right)^2-1\right|^2=C s^2,
\label{Bsc}
\end{equation}
for some constant $C$, so an easy application of Cauchy-Schwarz inequality yields
\begin{equation}
\begin{split}
 &\left|\frac{\partial^4}{\partial x_i^4}\EP f(\Xx(s)) p_t(0,s)\right| \\
 &\le
 \frac{1}{\sqrt{2\pi t}} e^{-s^2/2t}\left(\EP\left|\frac{(x_i-X_i^x(s))^2-s}{s^2}\right|^2\EP\left|h_i(\Xx(s))-h_i(x)\right|^2\right)^{1/2} \\
 &\le \frac{K}{\sqrt{2\pi t}}\frac{e^{-s^2/2t}}{s}\left(\EP\left|\Xx(s)-x\right|^{2\alpha}\right)^{1/2}
 = \frac{K}{\sqrt{2\pi t}}\frac{e^{-s^2/2t}}{s^{1-{\alpha/2}}},
\end{split}
\label{tstep}
\end{equation}
where the next to last inequality follows from \eqref{Bsc} and the H\"{o}lder condition on $h_i$.  But
\begin{equation}
\frac{K}{\sqrt{2\pi t}}\int_0^\infty \frac{e^{-s^2/2t}}{s^{1-{\alpha/2}}} ds <\infty,
\label{absint}
\end{equation}
for $\alpha>0$.  So, \eqref{toshow} (a), as well as the continuity assertion afterwards, follow by a standard classical argument
(see e.g. Friedman (1964), pages 10-12 for the purely analytical details in the second order case).

Now, by part (a) of \eqref{toshow}, part (b) of \eqref{toshow} is established once we show
\begin{equation}
\frac{\partial^4}{\partial x_i^4}\int_0^t\int_0^\infty \EP g(\Xx(s)) p_r(0,s) dsdr=
\int_0^t\frac{\partial^4}{\partial x_i^4}\int_0^\infty \EP g(\Xx(s)) p_r(0,s) dsdr.
\label{toshow2}
\end{equation}
This is simple, however, since by the first part we have
\begin{equation}
\begin{split}
\left|\frac{\partial^4}{\partial x_i^4}\int_0^\infty\EP g(\Xx(s)) p_r(0,s)ds\right|&=
\left|\int_0^\infty\frac{\partial^4}{\partial x_i^4}\EP g(\Xx(s)) p_r(0,s)ds\right|\\
&\le\int_0^\infty\left|\frac{\partial^4}{\partial x_i^4}\EP g(\Xx(s)) p_r(0,s)\right| ds\\&\le
\frac{K}{\sqrt{2\pi r}}\int_0^\infty \frac{e^{-s^2/2r}}{s^{1-{\alpha/2}}} ds,
\end{split}
\label{finlem}
\end{equation}
and
$$\int_0^t \left(\frac{K}{\sqrt{2\pi r}}\int_0^\infty \frac{e^{-s^2/2r}}{s^{1-{\alpha/2}}} ds\right)dr<\infty.$$
So, again standard arguments (e.g., Problem 3 p.~52 in Ash (1972) or Friedman again) complete the proof.
\end{proof}
\begin{rem}
By an identical argument to that in \lemref{diffunderint} above, with only notational differences to accommodate the dependence on
time $r$ in $g(r,x)$, $\Delta^2$ can be pulled outside the integrals in  \eqref{ur} once $D_{ij} g$ is continuous on $[0,T]\times\Rd$ and
H\"older continuous in $x$ uniformly with respect to $(t,x)\in[0,T]\times\Rd$.
\label{DI}
\end{rem}
\subsection{The Main Proofs}
In all the proofs presented here it suffices to prove the result for the BTP case.  The excursion BTP (including the IBM) cases are proved from
the BTP one in exactly the same way as in the proof of Theorem 0.1 in \cite{AZ}.
\begin{proof}[Proof of \thmref{BTPPDES1}]
We first use the independence of $\Xx(\cdot)$ and $|B(\cdot)|$ to get
\begin{equation}
\begin{split}
\utx&=\EP\left[f(\BTP)+\int_0^tg(\BTPr)dr\right]\\
&= 2 \int_0^\infty \SG_s f(x) p_t(0,s) ds +  2 \int_0^\infty\int_0^t\SG_s g(x)p_r(0,s) dr ds
\end{split}
\label{expc}
\end{equation}
where $p_t(0,s)$ is the transition density of $B(t)$, and where we used the boundedness of $g$ and consequently that of $\SG_s g(x)$ along with Fubini's theorem to get the last term.
Differentiating $\eqref{expc}$ with respect to $t$ and putting the derivative under the integral,
which is easily justified by the dominated convergence theorem (remember that $f$ and $g$ are bounded), then using the fact that $p_t(0,s)$ satisfies the heat equation
$$\frac{\partial}{\partial t} p_t(0,s)=\frac12 \frac{\partial^2}{\partial s^2}p_t(0,s)$$ we have
\begin{equation}
\begin{split}
\frac{\partial}{\partial t} \utx & = 2\int_0^\infty \SG_s f(x) \frac{\partial}{\partial t}p_t(0,s) ds +2 \int_0^\infty\SG_s g(x)\left(\int_0^t \frac{\partial}{\partial r}p_r(0,s) dr\right)ds\\
&= \int_0^\infty\SG_s f(x)\frac{\partial^2}{\partial s^2} p_t(0,s) ds + \int_0^t \int_0^\infty\SG_s g(x)\frac{\partial^2}{\partial s^2}p_r(0,s) dsdr
\end{split}
\label{tder}
\end{equation}
 So, integrating by parts twice and observing that the boundary terms always vanish at $\infty$ (as $s\nearrow\infty$) and that
$\dstyle{(\partial/\partial s) p_t (0,s)}|_{ s=0}=0$ but $p_t(0,0) >0$, we get
\begin{equation*}
\begin{split}
\frac{\partial}{\partial t} \utx&= p_t (0,0)  \left.\left( {\frac {\partial }{\partial s}}\SG_s f(x)\right)\right|_{ s=0}  +
\int_0^t p_r (0,0) \left.\left( {\frac {\partial }{\partial s}}\SG_s g(x)\right)\right|_{ s=0}dr\\&  +
\int_0^\infty \GN^2\SG_s f(x)p_t (0,s)ds+\int_0^t \int_0^\infty\GN^2\SG_s g(x)p_r(0,s) dsdr\\
&=\frac{1}{\sqrt{2\pi t}}\GN f(x)+\frac{\sqrt{2 t}}{\pi}\GN g(x)\\&+\GN^2\left(\int_0^\infty \SG_s f(x)p_t (0,s)ds+
\int_0^t \int_0^\infty\SG_s g(x)p_r(0,s) dsdr\right)\\
&= \frac{1}{\sqrt{2\pi t}}\GN f(x)+\frac{\sqrt{2 t}}{\pi}\GN g(x)+\frac12\GN^2\utx,
\end{split}
\end{equation*}
where in the next to last step we used property \eqref{P}.  Obviously, $u(0,x)=f(x)$.  Also, rewriting $u$ as $\utx=\EP f(x+\BTPz)$
and noticing that $f$ is bounded and continuous, we can use the bounded convergence theorem to conclude
$$f(x)={\dstyle\lim_{\substack{t\downarrow0\\y\to x}}} u(t,y).$$
\end{proof}
Next, we present the
\begin{proof}[Proof of \thmref{PDE2th}]
Again, it's enough to prove the BTP case.  Let
\begin{equation}
\utxe\eqdef\EP\left[f(\eBTP)\exp\left(-{\frac{|B(t)|}{\epsilon}}\right)\right],
\label{exp}
\end{equation}
and
\begin{equation}
\vsxe\eqdef\EP\left[f(\Xx(\epsilon s))\exp\left(-{\frac{s}{\epsilon}}\right)\right]=\exp\left(-{\frac{s}{\epsilon}}\right)\SG_{\epsilon s}f(x).
\label{complexv1}
\end{equation}
We then have
\begin{equation}
\utxe=2\int_0^\infty\vsxe p_t(0,s)ds,
\label{uep}
\end{equation}
and so, following our argument in the previous proof; and noticing that, for a fixed $\epsilon$, $\exp\left(-{{s}/{\epsilon}}\right)$ and all of its derivatives are bounded
and in $C^\infty(\Rp;\Rp)$ we get,
\begin{equation}
\begin{split}
\frac{\partial}{\partial t} \utxe & = 2\int_0^\infty \vsxe\frac{\partial}{\partial t}p_t(0,s) ds= \int_0^\infty\vsxe\frac{\partial^2}{\partial s^2} p_t(0,s) ds\\
&=p_t (0,0)  \left.\left( {\frac {\partial }{\partial s}}\vsxe\right)\right|_{ s=0} +\int_0^\infty p_t(0,s)\frac{\partial^2}{\partial s^2} \vsxe  ds\\
&=\frac{1}{\sqrt{2\pi t}}\left[\epsilon\,\GN f(x)-\frac{1}{\epsilon}f(x)\right] + \frac{1}{2\epsilon^2}\utxe-\GN \utxe+\frac{\epsilon^2}{2}\GN^2\utxe,
\end{split}
\label{tder77}
\end{equation}
where we have also used property \eqref{P}.  Again, $u_\epsilon(0,x)=f(x)$.  Rewriting $u$ as $\utx=\EP\left[f(x+\BTPez)\exp\left(-{{|B(t)|}/{\epsilon}}\right)\right]$ and noticing that $f$ is bounded and continuous, we can use the bounded
convergence theorem to conclude
$$f(x)={\dstyle\lim_{\substack{t\downarrow0\\y\to x}}} u_\epsilon(t,y).$$
\end{proof}
We are now in a position to give the
\begin{proof}[Proof of \thmref{FK}]
Let $u$ and $v$ be defined as in \eqref{udefn1} and \eqref{vdefn1}, respectively.  Then, as we did several times above
\begin{equation}
\utx= 2 \int_0^\infty p_t(0,s)\vsx ds
\label{expc2}
\end{equation}
Differentiating $\eqref{expc2}$ with respect to $t$ and putting the derivative under the integral,
which is again justified by the dominated convergence theorem (remember that $f$ is bounded and $c\le 0$), then proceeding as in the proof of \thmref{BTPPDES1}
and note that, since $\Xx$ is Brownian motion starting at $x$ under $\P$, then
 $$\frac{\partial}{\partial s}\vsx=\frac12\Delta\vsx+c(x)\vsx,\mbox{ in }(0,\infty)\times\Rd,$$
 with $v$ continuous on $[0,\infty)\times\Rd$ and $v(0,x)=f(x)$ (see for example section 4.3 of \cite{D}).
 So, integrating by parts twice; and again observing that the boundary terms always vanish at $\infty$ (as $s\nearrow\infty$) and that
$\dstyle{(\partial/\partial s) p_t (0,s)}|_{ s=0}=0$ but $p_t(0,0) >0$, we get
\begin{equation}
\begin{split}
\frac{\partial}{\partial t} \utx&= - \int_0^\infty\frac{\partial}{\partial s} p_t (0,s) \frac{\partial}{\partial s} \vsx ds\\
&= p_t (0,0)  \left.\left( {\frac {\partial }{\partial s}}v \left( s,x \right) \right)\right|_{ s=0} + \int_0^\infty p_t (0,s) \frac{\partial^2}{\partial s^2}\vsx ds\\
&=p_t (0,0) \left(\frac12\Delta f(x)+c(x)f(x)\right)\\&+ \int_0^\infty p_t (0,s) \left(\frac14\Delta^2 v \left( s,x \right)
+\frac12 v \left( s,x \right)\Delta c \left( x \right)   \right) ds\\
 &+\int_0^\infty p_t (0,s) \left(\nabla c \left( x \right)\cdot \nabla v \left( s,x \right) +
c \left( x \right) \Delta v \left( s,x \right) +  c^{2} \left( x \right)   v \left( s,x \right) \right)ds
\end{split}
\label{FKGPDE}
\end{equation}
Taking the application of $\Delta$, $\nabla$, $\Delta^2$ as well as the terms $c$, $c^2$, $\Delta c$, $\nabla c$ and the dot product outside the integral;  we
get the PDE in \eqref{BTPPDE}.  To justify this last step, it suffices to show that we can take the highest order derivatives ($\Delta^2$) outside the integral.
Towards this end, we first note that
\begin{equation}
v(s,x)=\EP f(\Xx(s))+\int_0^s\EP \left\{c(\Xx(r))v(s-r,\Xx(r))\right\}dr.
\label{vr}
\end{equation}
This follows from exactly the same steps as those in Durrett's \cite{D} p.~140--141.   Then,
\begin{equation}
D_{ij} v(s,x)=D_{ij} \EP f(\Xx(s))+D_{ij}\left(\int_0^s\EP \left\{g(s-r,\Xx(r))\right\}dr\right); \ 1\le i,j\le d,
\label{lvr}
\end{equation}
where $g(r,x)=c(x)v(r,x)$.   Fix an arbitrary pair $i,j$ and let
$$v_1(s,x)\eqdef D_{ij} \EP f(\Xx(s)) \mbox{ and }v_2(s,x)\eqdef  D_{ij}\int_0^s\EP \left\{g(s-r,\Xx(r))\right\}dr.$$
We see from the boundedness of $f$ and the H\"older and boundedness assumptions on $D_{ij} f$ that
\begin{equation}
\begin{split}
\left|v_1(s,x)-v_1(s,y)\right|&=\left|\EP\left[D_{ij} f(x+X^0(s))-D_{ij} f(y+X^0(s))\right]\right|\\
&\le \EP\left|D_{ij} f(x+X^0(s))-D_{ij} f(y+X^0(s))\right|\le C\left|x-y\right|^\alpha.
\end{split}
\label{Hol1}
\end{equation}
In fact, the boundedness of $f$ and Problem 3.1 in \cite{KS} imply that  $D_{ij} \EP f(\Xx(s))=\EP D_{ij} f(x+X^0(s))$ has derivatives of all orders
(in both $s$ and $x$, for $(s,x)\in(0,\infty)\times\Rd$),  and hence is Lipschitz in $x$.

Now, the boundedness of $v$ (implied by the boundedness of $f$ and the fact that $c\le0$) and the boundedness of $c$ imply the boundedness of $g$.
This, in addition to Theorem 2.6c in \cite{D} (in Chapter 4) yield
\begin{equation}
\begin{split}
v_2(s,x)=\int_0^s\int_{\Rd}g(s-r,y)D_{ij} p^{(d)}_r(x,y)drdy.
\end{split}
\label{v2}
\end{equation}
But $|D_{ij} v(s,x)|\le K_T$ $\forall (s,x)\in[0,T]\times\Rd$, for any time $T>0$, by assumption; and $D_{ij}v $ is continuous by Theorem 3.6 p.~140 in
\cite{D} (since $f$ and $c$ are both bounded by assumption and H\"older continuous because $D_{ij} f$ and $D_{ij} c$ are by assumption).  This, in addition to
the assumption that $D_{ij} c$ are all H\"older continuous and bounded imply that $D_{ij} g(s,x)$ is bounded on $[0,T]$ uniformly in $x$ and continuous.
So that, if $G(r,y)\eqdef D_{ij} g(r,y)$, then \eqref{v2} implies that
\begin{equation}
 v_2(s,x)=\int_0^s\int_{\Rd}p^{(d)}_r(x,y) D_{ij} g(s-r,y)drdy= \int_0^s\int_{\Rd}p^{(d)}_r(x,y) G(s-r,y)drdy,
\label{v22}
\end{equation}
and it follows by Theorem 2.6b p.~133-134 in \cite{D} (in which the assertion of differentiability and continuity of derivatives is unaffected if we replace
$|G|\le M$ with $|G(t,x)|\le M_T$ $\forall(t,x)\in[0,T]\times\Rd,\ \forall T$),  that
\begin{equation}
\left|v_2(s,x)-v_2(s,y)\right|\le C_T|x-y|, \forall s\in[0,T].
\label{Lipv2}
\end{equation}
Clearly, \eqref{Hol1} and \eqref{Lipv2} imply that $D_{ij} v(s,x)=v_1(s,x)+v_2(s,x)$ is H\"older continuous with exponent $\alpha$ in $x$ uniformly with
respect to $(s,x)\in[0,T]\times\Rd$.  So, by the boundedness assumptions on $ D_{ij} v$ and $D_{ij} c $ and the assumption that $D_{ij} c$ is H\"older continuous,
it follows that $D_{ij} g(s,x)$ is H\"older continuous with exponent $\alpha$ in $x$ uniformly with respect to $(s,x)\in[0,T]\times\Rd$.

Now, since $f,g$ and $D_{ij} f,D_{ij} g$ are bounded and H\"older continuous, then \eqref{vr}, \lemref{diffunderint} and \remref{DI} imply that
\begin{equation}
\begin{split}
&\Delta^2  \int_0^\infty p_t (0,s) v \left( s,x \right)ds\\&=\Delta^2\int_0^\infty p_t(0,s) \left(\EP f(\Xx(s))+\int_0^s\EP \left\{g(s-r,\Xx(r))\right\}dr\right) ds\\
&=\Delta^2\int_0^\infty p_t (0,s)\EP f(\Xx(s)) ds+\Delta^2\int_0^\infty p_t (0,s)\int_0^s\EP \left\{g(s-r,\Xx(r))\right\}drds\\
&=\int_0^\infty p_t (0,s)\Delta^2\EP f(\Xx(s)) ds+\int_0^\infty p_t (0,s)\int_0^s\Delta^2\EP \left\{g(s-r,\Xx(r))\right\}drds\\
&=\int_0^\infty p_t (0,s)\left(\Delta^2\EP f(\Xx(s))+\Delta^2\int_0^s\EP \left\{g(s-r,\Xx(r))\right\}dr\right)ds\\&=\int_0^\infty p_t (0,s) \Delta^2 v \left( s,x \right)ds.
\end{split}
\label{ur}
\end{equation}
So, we can pull the operator $\Delta^2$ outside the integral in \eqref{FKGPDE} as desired.

Obviously, $u(0,x)=f(x)$, and rewriting $u$ as
$$\utx=\EP\left[f(x+\BTPz)\exp\left(\int_0^{|B(t)|}c(x+X^0(r))dr\right)\right]$$
and noticing that $f$ is bounded and continuous and $c\le0$, bounded, and continuous, we can use the bounded convergence theorem to conclude
$$f(x)={\dstyle\lim_{\substack{t\downarrow0\\y\to x}}} u(t,y).$$so that $\utx$ solves \eqref{BTPFKPDE}.
\end{proof}
\textbf{Acknowledgement} I'd like to thank the referee for his nice and constructive comments.

\end{document}